\documentclass[11pt,a4paper]{article}
\usepackage[left=2.0cm, top=2.45cm,bottom=2.45cm,right=2.0cm]{geometry}
\usepackage{mathtools,amssymb,amsthm,mathrsfs,calc,graphicx,xcolor,cleveref,dsfont,tikz,pgfplots,bm}
\usepackage[british]{babel}
\usepackage{amsfonts}              
\usepackage[T1]{fontenc}

\usepackage{verbatim}

\numberwithin{equation}{section}

\newtheorem {theorem}{Theorem}[section]
\newtheorem {proposition}[theorem]{Proposition}
\newtheorem {lemma}[theorem]{Lemma}

{\theoremstyle{definition}

}
{\theoremstyle{theorem}

}

\newcommand{\dint}{\textup{d}}

\newcommand{\vol}{\mathrm{vol}}

\def\BB{\mathbb{B}}

\def\EE{\mathbb{E}}

\def\NN{\mathbb{N}}

\def\RR{\mathbb{R}}
\def\SS{\mathbb{S}}

\def\bC{\mathbf{C}}

\def\bN{\mathbf{N}}

\def\bP{\mathbf{P}}

\def\bU{\mathbf{U}}

\def\bW{\mathbf{W}}

\def\cD{\mathcal{D}}

\begin{document}

\title{\bfseries Large deviations for uniform projections\\ of $p$-radial distributions on $\ell_p^n$-balls}

\author{Tom Kaufmann\footnotemark[1],\;\; Holger Sambale\footnotemark[2],\;\; and Christoph Th\"ale\footnotemark[3]}


\date{}
\renewcommand{\thefootnote}{\fnsymbol{footnote}}
\footnotetext[1]{Ruhr University Bochum, Germany. Email: tom.kaufmann@rub.de}

\footnotetext[2]{Ruhr University Bochum, Germany. Email: holger.sambale@rub.de}

\footnotetext[3]{Ruhr University Bochum, Germany. Email: christoph.thaele@rub.de}

\maketitle

\begin{abstract} \noindent We consider products of uniform random variables from the Stiefel manifold of orthonormal $k$-frames in $\mathbb{R}^n$, $k \le n$, and random vectors from the $n$-dimensional $\ell_p^n$-ball $\mathbb{B}_p^n$ with certain $p$-radial distributions, $p\in[1,\infty)$. The distribution of this product geometrically corresponds to the projection of the $p$-radial distribution on $\mathbb{B}^n_p$ onto a random $k$-dimensional subspace. We derive large deviation principles (LDPs) on the space of probability measures on $\RR^k$ for sequences of such projections.\\[2mm]
{\bf Keywords}. Large deviation principle, $\ell_p^n$-ball, random projection, Stiefel manifold \\
{\bf MSC:} 52A23, 
60F10, 

\end{abstract}

\section{Introduction}

The study of high-dimensional convexity goes back to studying infinite-dimensional normed spaces via their local structures, such as their unit balls, but has since become of substantial interest in its own right. The concentration phenomena exhibited by convex objects in high dimensions, analyzed in the language of probability, are of great use in applications such as compressed sensing, information theory and approximation theory (see \cite{chafai2012interactions, foucartmathematical, hinrichs2021random, hinrichs2021random2, krieg2020random}). Analogues of many well known limit results from probability have been found in high-dimensional convexity, such as the central limit theorem (CLT) of Klartag \cite{KlartagCLT, KlartagCLT2}, and recently, starting with the work of Gantert, Kim, and Ramanan \cite{GKR}, large deviations began to be considered with increasing interest as well. For the sake of brevity, we refer the reader to the classic literature on large deviations theory for more details on basic definitions and results \cite{DZ}.

Especially the $n$-dimensional $\ell_p^n$-ball $\mathbb{B}_p^n :=\{x \in \RR^n: \|x\|_p \le 1\}$, for $p \in [1, \infty)$ and $\lVert x \rVert_p := \big( \sum_{i=1}^n |x_i|^p \big)^{1/p}$, has been extensively considered in this regard (see, e.g., \cite{APTldp, KPTSanov, KPTLimitThm, KPTLimitThm2, KP-Stiefel}) due to both its relevance in geometry and functional analysis, and its accessibility via useful probabilistic representation results, which will be outlined below. For an overview of classical and more recent results about $\ell_p^n$-balls, we refer to the survey article \cite{PTTSurvey}.\\ 
The setting of the present work is a generalization of the one initiated by Kabluchko and Prochno \cite{KP-Stiefel} and can be described as follows. For $k \le n$, the Stiefel manifold $\mathbb{V}_{n,k}$ is the set of all orthonormal $k$-frames in $\mathbb{R}^n$, i.\,e., the set of all $k$-tuples of orthonormal vectors $v_1, \ldots, v_k$ in $\mathbb{R}^n$. Arranging these vectors into a $k \times n$ matrix $V$ with rows $v_1^T, \ldots, v_k^T$, we have the identification
\[
\mathbb{V}_{n,k} = \{V \in \mathbb{R}^{k \times n} \colon  VV^T = I_k\},
\]
where $I_k$ denotes the $k \times k$ identity matrix. We equip $\mathbb{V}_{n,k}$ with the uniform distribution (i.\,e., the invariant Haar probability measure) $\mu_{n,k}$, writing $V_{n,k}$ for the corresponding random variable. Recall that $V_{n,k}$ is characterized by the following invariance property: for any orthogonal matrices $O \in \mathbb{R}^{k \times k}$ and $O' \in \mathbb{R}^{n \times n}$, $OV_{n,k}O'$ has the same distribution as $V_{n,k}$.

In particular, for random vectors $X^{(n)}$ taking values in $\mathbb{R}^n$, we may regard $V \in \mathbb{V}_{n,k}$ as a linear map $V \colon \mathbb{R}^n \to \mathbb{R}^k$ and study the distribution of the vectors $VX^{(n)} \in \mathbb{R}^k$, which we denote by
$$\mu_{VX^{(n)}}(A) := \mathbb{P}(VX^{(n)} \in A)$$
for any Borel set $A\subseteq \RR^k$. In addition, we may also choose $V_{n,k} \in \mathbb{V}_{n,k}$ at random according to the uniform distribution $\mu_{n,k}$ on $\mathbb{V}_{n,k}$. In this case, the distribution of $V_{n,k} X^{(n)}$, which we denote by%
\begin{equation}\label{eq:DefTargetMeasure}
\mu_{V_{n,k} X^{(n)}}(A) := \mathbb{P}(V_{n,k} X^{(n)} \in A),
\end{equation}
is a random probability measure on $\RR^k$, that is a random variable taking values in the space $\mathcal{M}_1(\mathbb{R}^k)$ of probability measures on $\mathbb{R}^k$, which we equip with the topology of weak convergence. This can geometrically be interpreted as the projection of the distribution of the random vector $X^{(n)}$  onto a uniform random $k$-dimensional subspace.

We are interested in large deviation principles (LDPs) for the random probability measures $\mu_{V_{n,k} X^{(n)}}$, where $X^{(n)} \in \mathbb{B}_p^n$ with a distribution taken from the class of $p$-radial distributions introduced by Barthe, Gud\'eon, Mendelson and Naor in \cite{BartheGuedonEtAl}. Kabluchko and Prochno \cite{KP-Stiefel} gave very general LDPs for random matrices in the orthogonal group and the Stiefel manifold, and showed an LDP for $k$-dimensional projections of the special case of the \textit{uniform distribution} on $\mathbb{B}_p^n$ as an application. Based on \cite{KP-Stiefel}, our work largely extends the set of projected distributions for which such an LDP is shown from the uniform distribution to the aforementioned class of $p$-radial distributions. In particular, we will see that the large deviation behaviour observed by Kabluchko and Prochno \cite{KP-Stiefel} is universal for a large class of probability measures on $\mathbb{B}_p^n$. Moreover, we shall describe geometrically motivated distributions on $\mathbb{B}_p^n$ for which the LDP needs a suitable modification we also provide. We should also delineate the present work from the results shown by Kim and Ramanan in \cite[Theorems 2.4 \& 2.6]{KimRamanan-Stiefel}, who have shown, among other results, LDPs for uniform random projections of uniform random vectors in $\mathbb{B}^n_p$ onto $k$-dimensional subspaces. By the same arguments as put forth in \cite{KP-Stiefel}, we note that while the settings are quite similar, the key difference is in the object of study, which in \cite{KimRamanan-Stiefel} is the projection point itself, hence yielding an LDP on $\RR^k$, whereas in both \cite{KP-Stiefel} and this work it is the projected distribution on $\RR^k$, thus the main result yields an LDP on the space $\mathcal{M}_1(\mathbb{R}^k)$ of probability measures on $\mathbb{R}^k$.

\medskip

In the next section we briefly list the notation and background material we will need to formulate our theorems, which in turn are presented in Section \ref{sec:MainResults}. Section \ref{sec:Proofs} will then contain their respective proofs.

\section{Preliminaries} \label{sec:Preliminaries}

Let us first define the objects and distributions needed for the main results. We write $\mathcal{B}(\mathbb{R}^n)$ for the Borel $\sigma$-algebra, $\langle \, \cdot\, , \, \cdot \, \rangle_2$ for the Euclidean scalar product, and $\vol_n(\cdot)$ for the Lebesgue measure on $\mathbb{R}^n$. As already mentioned in the introduction, for $p \in [1, \infty)$, $n\in\NN$, and $x=(x_1,\ldots,x_n)\in\RR^n$ we denote by
$\|x\|_p := \big(\sum\limits_{i=1}^n|x_i|^p\big)^{1/p}$
the $\ell_p^n$-norm of $x$ and by $\BB_p^n:=\{x\in\RR^n:\|x\|_p\leq 1\}$ and $\SS_p^{n-1}:=\{x\in\RR^n:\|x\|_p=1\}$ unit $\ell_p^n$-ball and unit $\ell_p^n$-sphere, respectively. We define the uniform distribution on $\BB^n_p$ and the cone probability measure on $\SS^{n-1}_p$ as 
$$\bU_{n,p}(\,\cdot\,) := {\vol_n(\,\cdot\,)\over \vol_n(\BB_p^n)}\qquad\text{and}\qquad\bC_{n,p}(\,\cdot\,) := {\vol_n(\{rx:r\in[0,1],x\in\,\cdot\,\})\over\vol_n(\BB_p^n)}.$$
Following \cite{BartheGuedonEtAl}, for a sequence $(\bW_n)_{n\in\NN}$ of Borel probability measures $\bW_n$ on $[0, \infty)$ we define the sequence of distributions
\begin{equation} \label{eq:DefPnpW}
\bP_{n,p,\bW_n} := \bW_n(\{0\})\bC_{n,p}+\Psi_n\bU_{n,p}
\end{equation}
on $\BB^n_p$, where $\Psi_n(x)=\psi(\|x\|_p)$, $x\in\BB_p^n$, is the $p$-radial density given by
$$
\displaystyle \psi_n(s) = \displaystyle {1\over p^{n/p}\Gamma\big({n\over p}+1\big)}{1\over (1-s^p)^{{n\over p}+1}}\bigg[\int_{(0,\infty)} w^{n\over p} \, e^{-\frac{1}{p}\big({s^p\over 1-s^p}\big) w} \, \bW_n(\dint w)\bigg],\qquad 0\leq s\leq 1.
$$
One can think of $\bW_n$ as indicating how probability mass is distributed $p$-radially within $\BB_p^n$.
The motivation behind this class of distributions is twofold. First, they encompass many relevant distributions on $\BB_p^n$. For instance, choosing $\bW_n \equiv \delta_0$ to be the Dirac measure at $0$, we have that $\bP_{n,p,\bW_n} \equiv  \bC_{n,p}$ and for $\bW_n \equiv \mathrm{Exp}(1)$, we have that $\bP_{n,p,\bW_n} \equiv  \bU_{n,p}$. For $m \in \NN$ choosing $\bW_n = \gamma(\frac{m}{p}, \frac{1}{p})$, i.e., a gamma distribution with shape $m/p$ and rate $1/p$, it can be shown that $\bP_{n,p,\bW_n}$ then corresponds to the projection of $\bC_{n+m,p}$ onto its first $n$ coordinates. An analogue correspondence is given for $\bW_n = \gamma(1 + \frac{m}{p}, \frac{1}{p})$ and the projection of $\bU_{n+m,p}$ onto its first $n$ coordinates, see \cite{BartheGuedonEtAl}. The second reason we consider the class of distributions $\bP_{n,p,\bW_n}$ specifically is a useful probabilistic representation result. For $p\in [1,\infty)$ we say a real-valued random variable $X$ has a $p$-generalized Gaussian distribution, denoted as $X \sim \bN_p$, if its distribution has Lebesgue density 
$$
\displaystyle f_p(x) := \frac{1}{2 \, p^{1/p} \, \Gamma\big(1+\frac{1}{p}\big)}\, e^{-{|x|^p}/{p}},  \qquad x\in\RR.$$
For $X \sim \bN_p$ and $r >0$ the $r$-th moment of $X$ is given by
\begin{equation} \label{eq:MomentXp}
\EE \left[X^r\right] = \Gamma \left( \frac{1 + r}{p}\right){\Gamma \left( \frac{1}{p}\right)}^{-1}.
\end{equation}
In particular, $\EE[X^r]<\infty$ for all $r>0$. Using this $p$-generalized Gaussian distribution, the following results from \cite[Theorem 3]{BartheGuedonEtAl} gives a way to represent a random vector $X^{(n)} \sim \bP_{n,p,\bW_n}$ in $\BB_p^n$ as a vector with i.i.d.\ $p$-generalized Gaussian coordinates, normalized by its norm and via $\bW_n$.

\begin{proposition}\label{prop:Barthe}
Let $n\in\NN$ and $p\in(0,\infty)$. Let $Z^{(n)} = (Z_1, \ldots, Z_n)$ be a random vector, where $Z_1,  \ldots, Z_n$ are i.i.d.\ with $Z_i \sim \bN_p$, and $W_n$ a random variable with distribution $\bW_n$ on $[0, \infty)$, independent of $Z^{(n)}$. Then the random vector
$$
Z^{(n)}\over (\|Z^{(n)}\|_p^p+W_n)^{1/p}
$$
has distribution $\bP_{n,p,\bW_n}$ as in \eqref{eq:DefPnpW}.
\end{proposition}

 The final result presented in this section is one of the main results from \cite{KP-Stiefel}, namely Theorem D therein. It provides an LDP for random projections of product measures, which we will use in conjunction with the representation from Proposition \ref{prop:Barthe} to prove our theorems. In what follows we shall write $\cD(X)$ for the distribution of a random variable $X$. Moreover, let
 \[
 \mathcal{R}_2^{k \times \infty} := \{A = (A_{ij})_{i,j=1}^{k,\infty} \colon (A_{ij})_{j \in \mathbb{N}} \in \ell_2, i=1, \ldots, k\}
 \]
 be the set of all matrices $A \in \mathbb{R}^{k \times \infty}$ with square-summable rows. For $A \in \mathbb{R}^{k \times \infty}$ we denote by $\lVert AA^T \rVert_{\mathrm{op}}$ the operator norm of the matrix $AA^T \in \mathbb{R}^{k \times k}$, where the condition $A \in \mathcal{R}_2^{k \times \infty}$ guarantees that $AA^T$ is well-defined.
 
\begin{proposition}\label{prop:TheoremDKP}
Fix $k \in \NN$. For each $n \in \NN$ let $Z^{(n)} = (Z_1, \ldots, Z_n)$ be an $n$-dimensional random vector, where $Z_1, Z_2, \ldots$ are i.i.d.\ non-Gaussian random variables with symmetric distribution and finite moments of all orders. Let $\sigma^2 := \mathbb{E}[Z_1^2] >0$ be the variance of $Z_1$. Then, the sequence of random probability measures $\mu_{V_{n,k} Z^{(n)}}$, $n \ge k$, as in \eqref{eq:DefTargetMeasure} satisfies an LDP on $\mathcal{M}_1(\mathbb{R}^k)$ with speed $n$ and good rate function $I \colon \mathcal{M}_1(\mathbb{R}^k) \to [0, \infty]$ given by
\[
I(\nu) = - \frac{1}{2} \log \mathrm{det}(I_k - AA^T)
\]
if $\nu$ admits a representation of the form
\[
\nu = \mathcal{D}\Big(\sum_{j=1}^\infty A_{\bullet,j}Z_j + \sigma(I_k-AA^T)^{1/2}N_k\Big)
\]
for some matrix $A \in \mathcal{R}_2^{k \times \infty}$ with columns $A_{\bullet, 1}, A_{\bullet, 2}, \ldots$ such that $\lVert AA^T \rVert_{\mathrm{op}} < 1$, where $N_k$ is a $k$-dimensional standard Gaussian random vector independent of $Z_1, Z_2, \ldots$. If $\nu$ does not admit a representation of this form, $I(\nu) = \infty$.
\end{proposition}

 Note that the specific distribution of the $Z_i$ has a rather subtle influence on the rate function of the LDP via the matrix $A$ used in the representation of a given measure $\nu \in \mathcal{M}_1(\RR^k)$. As a side remark, note that in \cite[Theorem D]{KP-Stiefel}, the case of $\sigma^2 = 0$ actually has to be excluded. We have amended the result accordingly.  

\section{Main Results} \label{sec:MainResults}

We are now in the position to present the first of our main results for the projections of $p$-radial distributions $\bP_{n,p,\bW_n}$ on $\ell_p^n$-balls. In what follows we shall write $\overset{\cD}{=}$ for equality in distribution.

\begin{theorem}\label{thm:alpha}
Fix $p \in [1, \infty)$, $p \ne 2$, and $k \in \mathbb{N}$. Moreover, let $(\bW_n)_{n\in\NN}$ be a sequence of Borel probability measures on $[0,\infty)$ and $(W_n)_{n\in\NN}$ a sequence of random variables with $W_n \sim\bW_n$, such that $W_n/n \to \alpha \in [0, \infty)$ in probability. Finally, let $X^{(n)}, Y^{(n)}$ be random vectors in $\BB_p^n$ with $Y^{(n)} \sim \bP_{n,p,\bW_n}$ and $X^{(n)} \overset{\cD}{=} n^{1/p} \, Y^{(n)}$. Then, the sequence of random probability measures $\mu_{V_{n,k} X^{(n)}}$, $n \ge k$, as in \eqref{eq:DefTargetMeasure} satisfies an LDP on $\mathcal{M}_1(\mathbb{R}^k)$ with speed $n$ and good rate function $I \colon \mathcal{M}_1(\mathbb{R}^k) \to [0, \infty]$ given by
\[
I(\nu) = - \frac{1}{2} \log \mathrm{det}(I_k - AA^T)
\]
if $\nu$ admits a representation of the form
\[
\nu = \mathcal{D}\Big(\Big(\frac{1}{1+\alpha}\Big)^{1/p}\sum_{j=1}^\infty A_{\bullet,j}Z_j + \sigma_{p, \alpha}(I_k-AA^T)^{1/2}N_k\Big)
\]
for some matrix $A \in \mathcal{R}_2^{k \times \infty}$ with columns $A_{\bullet, 1}, A_{\bullet, 2}, \ldots$ such that $\lVert AA^T \rVert_{\mathrm{op}} < 1$, where
$Z_1, Z_2, \ldots$ are i.i.d. with $Z_i \sim \bN_p$,
\[
\sigma_{p, \alpha}^2 := \Big(\frac{p}{1+\alpha}\Big)^{2/p} \frac{\Gamma(3/p)}{\Gamma(1/p)},
\]
and $N_k$ is an independent $k$-dimensional standard Gaussian random vector. If $\nu$ does not admit a representation of this form, $I(\nu) = \infty$.
\end{theorem}

As discussed earlier, choosing $\bW_n \equiv \delta_0$ gives $\bP_{n,p,\bW_n} \equiv \bC_{n,p}$ and $\bW_n \equiv \mathrm{Exp}(1)$ yields $\bP_{n,p,\bW_n} \equiv  \bU_{n,p}$, in both cases it holds for $W_n \sim \bW_n$ that $W_n/n \to 0$ in probability and we get back \cite[Theorem C]{KP-Stiefel}. Hence, we can see that both $\bC_{n,p}$ and $\bU_{n,p}$ share the same LDP behaviour in high dimensions, which is in line with similar observations made for other functionals (see, e.g., \cite{APTldp, KPTSanov, kaufmann2021weighted}). Moreover, the result even implies a certain universality of the rate function, since despite the expected sensitivity of LDPs to the underlying distributions, the rate function is the same for all sequences $(\bW_n)_{n\in\NN}$ that share the same limiting behaviour. \\
\\
Given the setting of Theorem \ref{thm:alpha}, if we consider the case $W_n/n \to \infty$ in probability (formally corresponding to the choice $\alpha = \infty$), by the representation result in Proposition \ref{prop:Barthe} one can see that this corresponds to each component of $X^{(n)}$ converging to $0$ in probability, that is, we arrive at a trivial limit. To avoid this, we may choose a different scaling as carried out in the following theorem.
\begin{theorem}\label{thm:beta}
Fix $p \in [1, \infty)$, $p \ne 2$, and $k \in \mathbb{N}$. Moreover, let $(\bW_n)_{n\in\NN}$ be a sequence of Borel probability measures on $[0,\infty)$ and $(W_n)_{n\in\NN}$ a sequence of random variables with $W_n\sim \bW_n$ and $W_n/n^\kappa \to \beta \in (0, \infty)$ in probability for some $\kappa >1$, and assume that the sequence of random variables $(W_n/n^{\kappa})^{-2/p}$ is uniformly integrable. Finally, let $X^{(n)}, Y^{(n)}$ be random vectors in $\BB_p^n$ with $Y^{(n)} \sim \bP_{n,p,\bW_n}$ and $X^{(n)} \overset{\cD}{=} n^{\kappa/p} \, Y^{(n)}$. Then, the sequence of random probability measures $\mu_{V_{n,k} X^{(n)}}$, $n \ge k$, as in \eqref{eq:DefTargetMeasure} satisfies an LDP on $\mathcal{M}_1(\mathbb{R}^k)$ with speed $n$ and good rate function $I \colon \mathcal{M}_1(\mathbb{R}^k) \to [0, \infty]$ given by
\[
I(\nu) = - \frac{1}{2} \log \mathrm{det}(I_k - AA^T)
\]
if $\nu$ admits a representation of the form
\[
\nu = \mathcal{D}\Big(\Big(\frac{1}{\beta}\Big)^{1/p}\sum_{j=1}^\infty A_{\bullet,j}Z_j + \sigma_{p, \beta}(I_k-AA^T)^{1/2}N_k\Big)
\]
for some matrix $A \in \mathcal{R}_2^{k \times \infty}$ with columns $A_{\bullet, 1}, A_{\bullet, 2}, \ldots$ such that $\lVert AA^T \rVert_{\mathrm{op}} < 1$, where $Z_1, Z_2, \ldots$ are i.i.d.\ $p$-generalized Gaussian random variables,
\[
\sigma_{p, \beta}^2 := \Big(\frac{p}{\beta}\Big)^{2/p} \frac{\Gamma(3/p)}{\Gamma(1/p)},
\]
and $N_k$ is an independent $k$-dimensional standard Gaussian random vector. If $\nu$ does not admit a representation of this form, $I(\nu) = \infty$.
\end{theorem}

Note that a helpful sufficient condition for the uniform integrability of $(W_n/n)^{-2/p}$ is given by
\begin{equation}\label{eq:SuffCond}
\sup_{n\in\NN}\, \mathbb{E}\left[\Big(\frac{n^\kappa}{W_n}\Big)^{4/p}\right] \le C
\end{equation}
for some absolute constant $C > 0$. In particular, it can be applied to verify the uniform integrability for certain gamma distributions.

\begin{lemma}\label{lem:Moments}
Suppose for each $n \in \NN$ that $W_n $
follows a gamma distribution with shape $a_n$ and rate $b>0$, where $(a_n)_{n\in\NN}$ is a positive increasing sequence and $b>0$. We assume that $a_n$ satisfies $\inf_n a_n=:m>4/p$ and $\lim_{n\to\infty}{a_n\over n^\kappa}=\lambda\in(0,\infty)$ for some $\kappa\in(0,\infty)$. Then
\begin{equation} \label{eq:GammaNegMoment}
\sup_{n\in\NN}\, \mathbb{E}\left[\Big(\frac{n^\kappa}{W_n}\Big)^{4/p}\right] \leq b^{4/p}M_p(\lambda m)^{-4/p}<\infty,
\end{equation}
where $M_p^{-1}=\prod_{i=0}^4(1-{4\over p(m+i)})$.
\end{lemma}

The proof of this lemma is postponed to the end of this paper. As a concrete and geometrically motivated example we consider the distribution on $\BB_p^n$ arising as the projection to the first $n$ coordinates of the cone probability measure $\bC_{n+m_n,p}$ on $\BB_p^{n+m_n}$, where $m_n$ is an increasing sequence satisfying $\inf_n m_n=m>4$ and $\lim_{n\to\infty}{m_n\over n^\kappa}=\lambda$ for some $\kappa\geq 1$ and $\lambda\in(0,\infty)$. As discussed above, this case corresponds to $\bP_{n,p,\bW_n}$ with $\bW_n=\gamma({m_n\over p},{1\over p})$ and fits the assumptions of Lemma \ref{lem:Moments}. The same holds for the projection of the uniform distribution $\bU_{n + m_n,p}$ corresponding to $\bP_{n,p,\bW_n}$ with $\bW_n=\gamma(1+{m_n\over p},{1\over p})$. In particular, Theorem \ref{thm:beta} applies to these situations.

%
%

\section{Proofs}\label{sec:Proofs}

This section shall prove Theorem \ref{thm:alpha} and Theorem \ref{thm:beta}. The proofs will follow in the footsteps of the proof of \cite[Theorem C]{KP-Stiefel}, adapting and generalizing the arguments where necessary. We start off by formulating some probabilistic representations of the target quantities and show some auxiliary results for the proof.\\
\\
Assume the set-up of Theorem \ref{thm:alpha} and for a fixed Stiefel matrix $V \in \mathbb{V}_{n,k}$ denote by $V_{\bullet,j}$, $j=1, \ldots, n$ its columns. Then, by \eqref{eq:DefTargetMeasure} and the representation results from Proposition \ref{prop:Barthe}
it follows that for any Borel set $A \in \mathcal{B}(\RR^k)$,
\begin{equation}\label{eq:muV}
    \mu_{V X^{(n)}}(A) = \mathbb{P}(VX^{(n)} \in A) = \mathbb{P}\Big(\sum_{j=1}^n n^{1/p} \frac{Z_j}{(\lVert Z^{(n)} \rVert_p^p + W_n)^{1/p}} V_{\bullet,j} \in A\Big),
\end{equation}
where $Z^{(n)} = (Z_1, \ldots, Z_n)$ with $Z_j \sim \bN_p$ i.i.d.\ and $W_n \sim \bW_n$ independent of $Z^{(n)}$. Moreover, let
\begin{equation}\label{eq:mutValpha}
    \tilde{\mu}_{V X^{(n)}}(A) := \mathbb{P}\Big(\Big(\frac{1}{1+\alpha}\Big)^{1/p}\sum_{j=1}^n Z_j V_{\bullet,j} \in A\Big),
\end{equation}
again with i.i.d. $Z_j \sim \bN_p$. We shall see that we can confine our analysis to $\tilde{\mu}_{V X^{(n)}}$ instead of $\mu_{V X^{(n)}}$, since they are arbitrarily close to each other in $n\in  \NN$ with respect to the L\'evy-Prokhorov metric. On the space $\mathcal{M}_1(\RR^k)$ of probability measures on $\RR^k$, the L\'{e}vy--Prokhorov metric $\rho_\mathrm{LP}$ is defined by
\[
\rho_\mathrm{LP}(\mu,\nu) := \inf \{\varepsilon > 0 \colon \mu(A) \le \nu(A_\varepsilon) + \varepsilon \text{ and } \nu(A) \le \mu(A_\varepsilon) + \varepsilon \text{ for all $A \in \mathcal{B}(\RR^k)$}\},
\]
where $A_\varepsilon$ denotes the $\varepsilon$-neighborhood of $A \in \mathcal{B}(\RR^k)$, defined as
\[
A_\varepsilon := \{x \in \RR^k \colon \|a-x\|_2 < \varepsilon \text{ for some $a \in A$}\},\qquad\varepsilon>0.
\]
We shall now prove that for the L\'{e}vy--Prokhorov metric $\rho_\mathrm{LP}$ on $\mathcal{M}_1(\mathbb{R}^k)$, $\rho_\mathrm{LP}(\mu_V, \tilde{\mu}_V)$ converges to $0$ uniformly over all $V \in \mathbb{V}_{n,k}$, as $n \to \infty$.

\begin{lemma}\label{LPmualpha}
For $p\in [1, \infty)$ and any $n\in \NN$ set $X^{(n)}$ as in Theorem \ref{thm:alpha}. Then, for $k \le n$, we have
\[
\lim_{n \to \infty} \sup_{V \in \mathbb{V}_{n,k}} \rho_\mathrm{LP}(\mu_{V X^{(n)}}, \tilde{\mu}_{V X^{(n)}}) = 0.
\]
\end{lemma}

\begin{proof}
Let $A \in \mathcal{B}(\RR^k)$, and $\varepsilon > 0$. Then, 
\begin{align}
    \tilde{\mu}_{V X^{(n)}}(A)
    &= \mathbb{P} \Big(\Big(\frac{1}{1+\alpha}\Big)^{1/p}\sum_{j=1}^n Z_j V_{\bullet, j} \in A \Big)\notag\\
    &\le \mathbb{P}\Big(\sum_{j=1}^n n^{1/p} \frac{Z_j}{(\lVert Z^{(n)} \rVert_p^p + W_n)^{1/p}} V_{\bullet,j} \in A_\varepsilon\Big)\notag\\
    &\quad+ \mathbb{P}\Big(\Big\lVert \Big(\frac{1}{1+\alpha}\Big)^{1/p} \sum_{j=1}^n Z_j V_{\bullet, j} - \sum_{j=1}^n n^{1/p} \frac{Z_j}{(\lVert Z^{(n)} \rVert_p^p + W_n)^{1/p}} V_{\bullet,j} \Big\rVert_2 \ge \varepsilon\Big)\notag\\
    &= \mu_{V X^{(n)}}(A_\varepsilon) + \mathbb{P}\Big(\Big\lVert \Big(\frac{1}{1+\alpha}\Big)^{1/p} \sum_{j=1}^n Z_j V_{\bullet, j} - \sum_{j=1}^n n^{1/p} \frac{Z_j}{(\lVert Z^{(n)} \rVert_p^p + W_n)^{1/p}} V_{\bullet,j} \Big\rVert_2 \ge \varepsilon\Big).\label{eq:ineq1alpha}
\end{align}
Let us prove that the second summand on the right-hand side converges to $0$, as $n \to \infty$. By Markov's inequality,
\begin{align*}
&\mathbb{P}\Big(\Big\lVert \Big(\frac{1}{1+\alpha}\Big)^{1/p} \sum_{j=1}^n Z_j V_{\bullet, j} - \sum_{j=1}^n n^{1/p} \frac{Z_j}{(\lVert Z^{(n)} \rVert_p^p + W_n)^{1/p}} V_{\bullet,j} \Big\rVert_2 \ge \varepsilon\Big)\\
&\hspace{2cm} \le \varepsilon^{-1} \mathbb{E} \Big\lVert \Big(\frac{1}{1+\alpha}\Big)^{1/p} \sum_{j=1}^n Z_j V_{\bullet, j} - \sum_{j=1}^n n^{1/p} \frac{Z_j}{(\lVert Z^{(n)} \rVert_p^p + W_n)^{1/p}} V_{\bullet,j} \Big\rVert_2,
\end{align*}
and by the Cauchy--Schwarz inequality,
\begin{align}
    &\mathbb{E} \Big\lVert \Big(\frac{1}{1+\alpha}\Big)^{1/p} \sum_{j=1}^n Z_j V_{\bullet, j} - \sum_{j=1}^n n^{1/p} \frac{Z_j}{(\lVert Z^{(n)} \rVert_p^p + W_n)^{1/p}} V_{\bullet,j} \Big\rVert_2\notag\\
    &\qquad= \mathbb{E} \Big(\Big\lVert \sum_{j=1}^n Z_j V_{\bullet, j}\Big\lVert_2 \cdot \Big| \Big(\frac{1}{1+\alpha}\Big)^{1/p} - \frac{n^{1/p}}{(\lVert Z^{(n)} \rVert_p^p + W_n)^{1/p}}\Big|\Big)\notag\\
    &\qquad\le \sqrt{\mathbb{E}\Big\lVert \sum_{j=1}^n Z_j V_{\bullet, j}\Big\lVert_2^2} \cdot \sqrt{\mathbb{E}\Big| \Big(\frac{1}{1+\alpha}\Big)^{1/p} - \frac{n^{1/p}}{(\lVert Z^{(n)} \rVert_p^p + W_n)^{1/p}}\Big|^2}.\label{eq:ineq2alpha}
\end{align}
As $Z_1, \ldots, Z_n$ are i.i.d.\ with mean zero and $V_{\bullet,1}, \ldots, V_{\bullet,n}$ are orthonormal vectors, the first factor in \eqref{eq:ineq2alpha} reads 
\begin{equation}\label{eq:L2eqalpha}
\mathbb{E}\Big\lVert \sum_{j=1}^n Z_j V_{\bullet, j}\Big\lVert_2^2
= \mathbb{E}\Big[ \sum_{i,j=1}^n Z_iZ_j \langle V_{\bullet, i}, V_{\bullet, j} \rangle_2 \Big]
= \mathbb{E}[Z_1^2] \sum_{j=1}^n \langle V_{\bullet, i}, V_{\bullet, j} \rangle_2
= k \mathbb{E}[Z_1^2].
\end{equation}

To address the second factor in \eqref{eq:ineq2alpha}, let us first argue that
\begin{equation}\label{LLNalpha}
    \xi_n := \Big(\Big(\frac{1}{1+\alpha}\Big)^{1/p} - \frac{n^{1/p}}{(\lVert Z^{(n)} \rVert_p^p + W_n)^{1/p}}\Big)^2 \longrightarrow 0
\end{equation}
in probability, as $n \to \infty$. Indeed, by the continuous mapping theorem, it suffices to show that
\[
\frac{\lVert Z^{(n)} \rVert_p^p}{n} + \frac{W_n}{n} \longrightarrow 1 + \alpha
\]
in probability. This follows from the fact that as $Z_1, \ldots, Z_n$ are i.i.d.\ $p$-generalized Gaussian random variables, we have $\mathbb{E}|Z_i|^p = 1$, and moreover that by assumption, $W_n/n \longrightarrow \alpha$ in probability. In fact, we even have $\xi_n \longrightarrow 0$ in $L^1$. To see this, it suffices to show that $(\xi_n)_n$ is uniformly integrable, which in combination with \eqref{LLNalpha} yields convergence in $L^1$. Clearly, $(\xi_n)_n$ is uniformly integrable if the sequence
\[
\Big(\frac{n}{\lVert Z^{(n)} \rVert_p^p + W_n}\Big)^{2/p} \le \Big(\frac{n}{\lVert Z \rVert_p^p}\Big)^{2/p}
 \]
is uniformly integrable, where we have used that $W_n \ge 0$. This in turn follows from the fact that $\|Z^{(n)}\|_p^p \sim \gamma(n/p, 1/p)$ together with \eqref{eq:GammaNegMoment} for $\bW_n = \gamma(n/p, 1/p)$, $\kappa=1$ and $\lambda = 1/p$, which yields 
\[
\mathbb{E}\Big(\frac{n}{\lVert Z^{(n)} \rVert_p^p}\Big)^{4/p} \le p^{4/p}M_p \in (0,\infty)
\]
for all $n\in\NN$. Hence, $\xi_n \longrightarrow 0$ in $L^1$, and as a consequence, the second factor in \eqref{eq:ineq2alpha} converges to $0$. This implies that the second summand in \eqref{eq:ineq1alpha} converges to $0$ uniformly in $V \in \mathbb{V}_{n,k}$. Altogether, we have proven that for any $\varepsilon>0$,
\[
\tilde{\mu}_{V X^{(n)}}(A) \le \mu_{V X^{(n)}}(A_\varepsilon) + \varepsilon
\]
for $n$ sufficiently large. In the same way, we may also prove that
\[
\mu_{V X^{(n)}}(A) \le \tilde{\mu}_{V X^{(n)}}(A_\varepsilon) + \varepsilon
\]
for $n$ sufficiently large, which finishes the proof.
\end{proof}

Finally, let us replace $V \in \mathbb{V}_{n,k}$ by random variables $V_{n,k}$, i.\,e.\ the Stiefel matrix is chosen at random according to the uniform distribution $\mu_{n,k}$ on $\mathbb{V}_{n,k}$. Based on Lemma \ref{LPmualpha}, we may prove that a weak LDP for the modified sequence $\tilde{\mu}_{V_{n,k} X^{(n)}}$ implies a weak LDP (in the sense of \cite[Definition, p.7]{DZ}) for $\mu_{V_{n,k} X^{(n)}}$, both respectively defined as in \eqref{eq:muV} and \eqref{eq:mutValpha} with respect to $V_{n,k}$.

\begin{lemma}\label{LDPsequivalpha}
Assume the set-up of Theorem \ref{thm:alpha} and recall the notation \eqref{eq:muV} and \eqref{eq:mutValpha}. If the sequence $\tilde{\mu}_{V_{n,k} X^{(n)}}$ satisfies a weak LDP on $\mathcal{M}_1(\mathbb{R}^k)$ at speed $n$ and rate function $I$, then the sequence $\mu_{V_{n,k} X^{(n)}}$ satisfies the same weak LDP.
\end{lemma}

\begin{proof}
It suffices to check the weak LDP on a basis of the topology of $\mathcal{M}_1(\mathbb{R}^k)$, e.\,g., the balls
\[
B_r(\nu) := \{\mu \in \mathcal{M}_1(\mathbb{R}^k) \colon \rho_\mathrm{LP}(\mu,\nu) < r\}
\]
for any $r \in (0, \infty)$. By Lemma \ref{LPmualpha}, for $n$ sufficiently large we have $\rho_\mathrm{LP}(\tilde{\mu}_{V_{n,k} X^{(n)}},\mu_{V_{n,k} X^{(n)}}) < r/2$ uniformly over all realizations of $V_{n,k} \in \mathbb{V}_{n,k}$. Therefore, by the triangle inequality for $\rho_\mathrm{LP}$,
\[
\frac{1}{n} \log\mathbb{P}(\tilde{\mu}_{V_{n,k} X^{(n)}} \in B_{r/2}(\nu))
\le \frac{1}{n} \log\mathbb{P}(\mu_{V_{n,k} X^{(n)}} \in B_r(\nu))
\le \frac{1}{n} \log\mathbb{P}(\tilde{\mu}_{V_{n,k} X^{(n)}} \in B_{3r/2}(\nu)),
\]
and hence,
\begin{align*}
\limsup_{n \to \infty} \frac{1}{n} \log\mathbb{P}(\tilde{\mu}_{V_{n,k} X^{(n)}} \in B_{r/2}(\nu))
&\le \liminf_{n \to \infty} \frac{1}{n} \log\mathbb{P}(\mu_{V_{n,k} X^{(n)}} \in B_r(\nu))\\
&\hspace{-2cm}\le \limsup_{n \to \infty} \frac{1}{n} \log\mathbb{P}(\mu_{V_{n,k} X^{(n)}} \in B_r(\nu))
\le \liminf_{n \to \infty} \frac{1}{n} \log\mathbb{P}(\tilde{\mu}_{V_{n,k} X^{(n)}} \in B_{3r/2}(\nu)).
\end{align*}
Thus, by monotonicity in $r$, taking the infimum over $r \in (0, \infty)$, the LDP for $\tilde{\mu}_{V_{n,k} X^{(n)}}$ yields
\[
-I(\nu) \le \inf_{r\in(0,\infty)} \liminf_{n \to \infty} \frac{1}{n} \log\mathbb{P}(\mu_{V_{n,k} X^{(n)}} \! \in \! B_r(\nu))
\le \inf_{r\in(0,\infty)} \limsup_{n \to \infty} \frac{1}{n} \log\mathbb{P}(\mu_{V_{n,k} X^{(n)}} \! \in \! B_r(\nu)) \le -I(\nu).
\]
From here the claim follows from \cite[Theorem 4.1.11]{DZ}.
\end{proof}

On a compact space, weak and full LDPs coincide. Here, compactness is provided by the following lemma.

\begin{lemma}\label{compalpha}
There is a constant $C \in (0,\infty)$ such that for all $n \ge k$ and all $V \in \mathbb{V}_{n,k}$,
\[
\mu_{V X^{(n)}} \in M_C := \Big\{\mu \in \mathcal{M}_1(\mathbb{R}^k) \colon \int_{\mathbb{R}^k} \lVert x \rVert_2 \,  \mu(\dint x) \le C\Big\},
\]
where the set $M_C$ is compact for any choice of $C \in (0,\infty)$.
\end{lemma}

\begin{proof}
The compactness of the set $M_C$ in the weak topology on $\mathcal{M}_1(\mathbb{R}^k)$ has been shown in \cite[Proof of Lemma 5.3]{KP-Stiefel}, so it remains to prove the first assertion. To this end, recalling the representation of the distribution $\mu_{V X^{(n)}}$ given in \eqref{eq:muV}, it suffices to prove that 
\[
\limsup_{n \to \infty} \sup_{V \in \mathbb{V}_{n,k}} \mathbb{E}\Big\lVert\sum_{j=1}^n n^{1/p} \frac{Z_j}{(\lVert Z^{(n)} \rVert_p^p + W_n)^{1/p}} V_{\bullet,j} \Big\rVert_2 < \infty,
\]
for i.i.d. $Z_j \sim \bN_p$ and $W_n \sim \bW_n$ as in Theorem \ref{thm:alpha}. By the triangle inequality it then follows that 
\begin{align*}
&\mathbb{E}\Big\lVert\sum_{j=1}^n n^{1/p} \frac{Z_j}{(\lVert Z^{(n)} \rVert_p^p + W_n)^{1/p}} V_{\bullet,j} \Big\rVert_2\\
&\le \mathbb{E} \Big\lVert \Big(\frac{1}{1+\alpha}\Big)^{1/p} \sum_{j=1}^n Z_j V_{\bullet, j} - \sum_{j=1}^n n^{1/p} \frac{Z_j}{(\lVert Z^{(n)} \rVert_p^p + W_n)^{1/p}} V_{\bullet,j} \Big\rVert_2 + \mathbb{E} \Big\lVert \Big(\frac{1}{1+\alpha}\Big)^{1/p} \sum_{j=1}^n Z_j V_{\bullet, j} \Big\rVert_2.
\end{align*}
The first summand on the right hand side converges to $0$ uniformly in $V_{n,k} \in \mathbb{V}_{n,k}$ as was shown after \eqref{eq:ineq2alpha}. Moreover, by Hölder's inequality and \eqref{eq:L2eqalpha}, the second summand is uniformly bounded by $\sqrt{k\mathbb{E}[Z_1^2]}/(1+\alpha)^{1/p}$, and thus, the claim follows.
\end{proof}

Combining the accumulated auxiliary results, we now have the sufficient tools to prove Theorem \ref{thm:alpha}.

\begin{proof}[Proof of Theorem \ref{thm:alpha}]
Apply Proposition \ref{prop:TheoremDKP} to the symmetric non-Gaussian random variables $Z_j/(1+\alpha)^{1/p}$, which, by \eqref{eq:MomentXp}, have finite moments of all orders and, in particular, variance
\[
\sigma_{p, \alpha}^2 = \Big(\frac{p}{1+\alpha}\Big)^{2/p} \frac{\Gamma(3/p)}{\Gamma(1/p)}.
\]
Hence, the sequence $\tilde{\mu}_{V_{n,k} X^{(n)}}$ satisfies an LDP on $\mathcal{M}_1(\mathbb{R}^k)$ with speed $n$ and rate function $I$ as stated in Theorem \ref{thm:alpha}. Therefore, by Lemma \ref{LDPsequivalpha}, $\mu_{V_{n,k} X^{(n)}}$ satisfies the same weak LDP, which extends to a full LDP by the compactness arguments given in Lemma \ref{compalpha}, thus finishing the proof.
\end{proof}

The proof of Theorem \ref{thm:beta} works in a very similar way to that of Theorem \ref{thm:alpha}, hence we will only point out the steps where it differs from the previous proof. Given the different scaling of $X^{(n)}$, it follows that for a Stiefel matrix $V \in \mathbb{V}_{n,k}$, we have 

\begin{equation}\label{eq:muVbeta}
    \mu_{V X^{(n)}}(A) := \mathbb{P}(VX^{(n)} \in A) = \mathbb{P}\Big(\sum_{j=1}^n n^{\kappa/p} \frac{Z_j}{(\lVert Z^{(n)}  \rVert_p^p + W_n)^{1/p}} V_{\bullet,j} \in A\Big)
\end{equation}
for any $A \in \mathcal{B}(\RR^k)$, and set 
\begin{equation}\label{mutValphabeta}
    \tilde{\mu}_{V X^{(n)}}(A) := \mathbb{P}\Big(\Big(\frac{1}{\beta}\Big)^{1/p}\sum_{j=1}^n Z_j V_{\bullet,j} \in A\Big),
\end{equation}
using the same notation as in Theorem \ref{thm:alpha} and its proof. The only argument that needs to be adapted is the proof of Lemma \ref{LPmualpha}, which will be replaced by the following Lemma.

\begin{lemma}\label{LPmualphaTilde}
For $p\in [1, \infty)$ and any $n\in \NN$ set $X^{(n)}$ as in Theorem \ref{thm:beta}. Then, for $k \le n$, we have
\[
\lim_{n \to \infty} \sup_{V \in \mathbb{V}_{n,k}} \rho_\mathrm{LP}(\mu_{V X^{(n)}}, \tilde{\mu}_{V X^{(n)}}) = 0.
\]
\end{lemma}

\begin{proof}
Let $A \in \mathcal{B}(\RR^k)$, and $\varepsilon > 0$. Then, by analogue expansion as in \eqref{eq:ineq1alpha}, we have that
\begin{align}
    \tilde{\mu}_{V X^{(n)}}(A)
    &\le \mu_{V X^{(n)}}(A_\varepsilon) + \mathbb{P}\Big(\Big\lVert \Big(\frac{1}{\beta}\Big)^{1/p} \sum_{j=1}^n Z_j V_{\bullet, j} - \sum_{j=1}^n n^{\kappa/p} \frac{Z_j}{(\lVert Z^{(n)}  \rVert_p^p + W_n)^{1/p}} V_{\bullet,j} \Big\rVert_2 \ge \varepsilon\Big).\label{ineq1alphatilde}
\end{align}
Again, we need to show that the second summand on the right-hand side in the above converges to $0$, as $n \to \infty$. By Markov's inequality, it holds that
\begin{align*}
&\mathbb{P}\Big(\Big\lVert \Big(\frac{1}{\beta}\Big)^{1/p} \sum_{j=1}^n Z_j V_{\bullet, j} - \sum_{j=1}^n n^{\kappa/p} \frac{Z_j}{(\lVert Z^{(n)}  \rVert_p^p + W_n)^{1/p}} V_{\bullet,j} \Big\rVert_2 \ge \varepsilon\Big)\\
&\hspace{2cm} \le \varepsilon^{-1} \mathbb{E} \Big\lVert \Big(\frac{1}{\beta}\Big)^{1/p} \sum_{j=1}^n Z_j V_{\bullet, j} - \sum_{j=1}^n n^{\kappa/p} \frac{Z_j}{(\lVert Z^{(n)}  \rVert_p^p + W_n)^{1/p}} V_{\bullet,j} \Big\rVert_2,
\end{align*}
and a further application of the Cauchy--Schwarz inequality as in \eqref{eq:ineq2alpha} yields
\begin{align}
    &\mathbb{E} \Big\lVert \Big(\frac{1}{\beta}\Big)^{1/p} \sum_{j=1}^n Z_j V_{\bullet, j} - \sum_{j=1}^n n^{\kappa/p} \frac{Z_j}{(\lVert Z^{(n)}  \rVert_p^p + W_n)^{1/p}} V_{\bullet,j} \Big\rVert_2\notag\\
    &\qquad\le \sqrt{\mathbb{E}\Big\lVert \sum_{j=1}^n Z_j V_{\bullet, j}\Big\lVert_2^2} \cdot \sqrt{\mathbb{E}\Big| \Big(\frac{1}{\beta}\Big)^{1/p} - \frac{n^{\kappa/p}}{(\lVert Z^{(n)} \rVert_p^p + W_n)^{1/p}}\Big|^2},\label{eq:ineq2alphatilde}
\end{align}
with the first factor simplifying to $k \mathbb{E}[Z_1^2]$ as in \eqref{eq:L2eqalpha}. It remains to show that
\begin{equation}\label{LLNalphatilde}
    \xi_n := \Big(\Big(\frac{1}{\beta}\Big)^{1/p} - \frac{n^{\kappa/p}}{(\lVert Z^{(n)} \rVert_p^p + W_n)^{1/p}}\Big)^2 \longrightarrow 0
\end{equation}
in probability, as $n \to \infty$, to address the second factor. We again do so by showing that
\[
\frac{\lVert Z^{(n)} \rVert_p^p}{n^\kappa} + \frac{W_n}{n^\kappa} \longrightarrow \beta
\]
in probability due to the continuous mapping theorem.
Since $\kappa >1$, by the same arguments as in the proof of Theorem \ref{thm:alpha}, it follows that $\lVert Z^{(n)} \rVert_p^p/n^\kappa \to 0$ and the behaviour of $W_n$ dominates. By assumption, $W_n/n^\kappa \longrightarrow \beta$ in probability. In fact, we even have $\xi_n \longrightarrow 0$ in $L^1$. Indeed, since the sequence of random variables $(W_n/n^{\kappa})^{-2/p}$ is uniformly integrable by assumption, it follows that the sequence of random variables $\xi_n$ is uniformly integrable as well, which which in combination with \eqref{LLNalphatilde} yields convergence in $L^1$.
As a consequence, the second factor in \eqref{eq:ineq2alphatilde} converges to $0$. This implies that the second summand in \eqref{ineq1alphatilde} converges to $0$ uniformly in $V \in \mathbb{V}_{n,k}$. Altogether, we have proven that for any $\varepsilon>0$,
\[
\tilde{\mu}_{V X^{(n)}}(A) \le \mu_{V X^{(n)}} (A_\varepsilon) + \varepsilon
\]
for $n$ sufficiently large and can prove analogously that
\[
\mu_{V X^{(n)}}(A) \le \tilde{\mu}_{V X^{(n)}}(A_\varepsilon) + \varepsilon
\]
for $n$ sufficiently large, thus finishing the proof.
\end{proof}

Since the rest of the proof of Theorem \ref{thm:alpha} does not depend on the specific choice of $\alpha$ or the scaling of the $X^{(n)}$, the remainder of the proof of Theorem \ref{thm:beta} can proceed in the very same way. It remains only to present the proof of Lemma \ref{lem:Moments}.

\begin{proof}[Proof of Lemma \ref{lem:Moments}]
We start by observing that
\begin{align*}
    \mathbb{E}\left[\Big(\frac{n^\kappa}{W_n}\Big)^{4/p}\right] = n^{4\kappa/p}{b^{a_n}\over\Gamma(a_n)}\int_0^{\infty}x^{a_n-1-4/p}e^{-bx}\,\dint x = n^{4\kappa/p}{b^{4/p}\over\Gamma(a_n)}\Gamma\Big(a_n-{4\over p}\Big).
\end{align*}
According to the inequality \cite[Equation (12)]{jameson2013inequalities} for quotients of gamma functions (applied with $x=a_n+5-4/p$ and $y=4/p$) one has that
\begin{align*}
 {\Gamma(a_n-{4\over p})\over\Gamma(a_n)}&=\Big(\prod_{i=0}^4{(a_n+i)\over(a_n-{4\over p}+i)}\Big){\Gamma(a_n-{4\over p}+5)\over\Gamma(a_n+5)}\\
 &= \Big({1\over\prod_{i=0}^4(1-{4\over p(a_n+i)})}\Big){1\over{\Gamma(a_n-{4\over p}+5)\over\Gamma(a_n+5)}}\\
 &\leq\Big({1\over\prod_{i=0}^4(1-{4\over p(m+i)})}\Big){1\over (a_n+4-{4\over p})^{4/p}}\\
 &\leq\Big({1\over\prod_{i=0}^4(1-{4\over p(m+i)})}\Big){1\over a_n^{4/p}} =M_p\cdot{1\over a_n^{4/p}},
\end{align*}
where we also used that $p\geq 1$. By our assumption on the growth of $a_n$ and since $a_n$ is increasing it follows that
\begin{equation*} 
\mathbb{E}\left[\Big(\frac{n^\kappa}{W_n}\Big)^{4/p}\right] \leq n^{4\kappa/p}b^{4/p}M_pa_n^{-4/p}\leq n^{4\kappa/p}b^{4/p}M_p(\lambda m n^\kappa)^{-4/p} = b^{4/p}M_p(\lambda m)^{-4/p}
\end{equation*}
for all $n\in\NN$. This completes the proof.
\end{proof}

\addcontentsline{toc}{section}{References}

\bibliographystyle{amsplain}

\providecommand{\bysame}{\leavevmode\hbox to3em{\hrulefill}\thinspace}
\providecommand{\MR}{\relax\ifhmode\unskip\space\fi MR }
\providecommand{\MRhref}[2]{%
	\href{http://www.ams.org/mathscinet-getitem?mr=#1}{#2}
}
\providecommand{\href}[2]{#2}

\end{document}